\documentclass[11pt]{article}
\usepackage[T1]{fontenc}
\usepackage[margin=1in]{geometry}

\usepackage{latexsym}
\usepackage[small,bf]{caption2}
\usepackage{graphics}
\usepackage{epsfig}
\usepackage{amsopn}
\usepackage{url}
\usepackage{xcolor}
\usepackage{hyperref}
\usepackage{algorithm}
\usepackage{algpseudocode}
\usepackage{graphics}
\usepackage{comment}

\usepackage[shortlabels]{enumitem}

\usepackage{multirow}

\usepackage{titling}

\usepackage{amsmath}
\usepackage{amssymb}
\usepackage{amsthm}
\usepackage{hyperref}   
\usepackage[nameinlink,noabbrev]{cleveref}
\usepackage{aliascnt}

\usepackage{mathtools}
\usepackage{enumitem}
\usepackage{bbm}
\usepackage{turnstile}

\newtheorem{theorem}{Theorem}[section]

\newaliascnt{lemma}{theorem}
\newtheorem{lemma}[lemma]{Lemma}
\aliascntresetthe{lemma}

\newaliascnt{proposition}{theorem}

\aliascntresetthe{proposition}

\newaliascnt{fact}{theorem}

\aliascntresetthe{fact}

\newaliascnt{claim}{theorem}

\aliascntresetthe{claim}

\newaliascnt{definition}{theorem}

\aliascntresetthe{definition}

\newaliascnt{example}{theorem}

\aliascntresetthe{example}

\newaliascnt{corollary}{theorem}

\aliascntresetthe{corollary}

\newaliascnt{assumption}{theorem}

\aliascntresetthe{assumption}

\newaliascnt{observation}{theorem}

\aliascntresetthe{observation}

\newaliascnt{conjecture}{theorem}
\newtheorem{conjecture}[conjecture]{Conjecture}
\aliascntresetthe{conjecture}

\newaliascnt{question}{theorem}
\newtheorem{question}[question]{Question}
\aliascntresetthe{question}

\newaliascnt{problem}{theorem}

\aliascntresetthe{problem}

\crefname{theorem}{theorem}{theorems}
\Crefname{theorem}{Theorem}{Theorems}
\crefname{lemma}{lemma}{lemmas}
\Crefname{lemma}{Lemma}{Lemmas}
\crefname{proposition}{proposition}{propositions}
\Crefname{proposition}{Proposition}{Propositions}
\crefname{fact}{fact}{facts}
\Crefname{fact}{Fact}{Facts}
\crefname{claim}{claim}{claims}
\Crefname{claim}{Claim}{Claims}
\crefname{definition}{definition}{definition}
\Crefname{definition}{Definition}{definitions}
\crefname{example}{example}{examples}
\Crefname{example}{Example}{Examples}
\crefname{corollary}{corollary}{corollaries}
\Crefname{corollary}{Corollary}{Corollaries}
\crefname{assumption}{assumption}{assumptions}
\Crefname{assumption}{Assumption}{Assumptions}
\crefname{observation}{observation}{observations}
\Crefname{observation}{Observation}{Observations}
\crefname{conjecture}{conjecture}{conjectures}
\Crefname{conjecture}{Conjecture}{Conjectures}
\crefname{question}{question}{questions}
\Crefname{question}{Question}{Questions}
\crefname{problem}{problem}{problems}
\Crefname{problem}{Problem}{Problems}

\renewcommand{\epsilon}{\varepsilon}

\newcommand{\rr}{\mathbb{R}}

\newcommand{\Pp}{\mathcal{P}}

\newcommand{\calF}{\mathcal{F}}

\newcommand{\calP}{\mathcal{P}}

\newcommand{\F}{\mathcal{F}}
\newcommand{\G}{\mathcal{G}}
\newcommand{\Hh}{\mathcal{H}}
\newcommand{\U}{\mathcal{U}}

\newcommand{\ignore}[1]{{}}

\usepackage{booktabs} 
\usepackage{tabularx}

\numberwithin{equation}{section}

\begin{document}

\title{The weighted union-set conjecture is false}
\author{Yimeng Wang \\ UCLA}

\maketitle

\date{\today}

\begin{abstract}
We disprove a weighted generalization of the union-set conjecture \cite{frankl1995extremal} raised in \cite{Gowers1}. We construct nonnegative weights on subsets of a finite ground set satisfying
$w(\varnothing)=0$ and $w(A\cup B)\ge\min\{w(A),w(B)\}$, while every
coordinate belongs to an arbitrarily small fraction of the total weight. 
\end{abstract}

\section{Introduction}\label{sec:intro}
Let $V$ be a finite universe. A family $\F\subseteq2^V$ is \emph{union-closed} if $A\cup B\in\F$ whenever
$A,B\in\F$. Frankl's union-closed sets conjecture asks whether every finite
union-closed family containing a nonempty set has an element in at least half
its members; see, for example, the discussion in~\cite{Gowers1,EIL}. There has been tremendous progress towards the conjecture in recent years. Most notably, using an entropy-based approach, Gilmer \cite{Gilmer2022} showed that there exists an element present in at least $c \geq 0.01$ fraction of the sets. Building upon this, Sawin \cite{Sawin2023} improved the constant to $c \geq 0.38$.

\begin{conjecture}[Frankl's union-set conjecture \cite{frankl1995extremal}]\label{conj:frankl}
    There exists a constant $c_{\mathrm{us}} \geq 1/2$ such that the following holds: fix a finite universe $V$. Let $\calF \subseteq 2^{V}$ be a union-closed set family containing at least one non-empty set. Then there exists an element $i \in V$ such that 
    \[
    |\{S \in \calF: i \in S\}| \geq c_{\mathrm{us}} |\calF|.
    \]
\end{conjecture}

In this paper, we consider a natural weighted analogue of this conjecture as discussed in \cite{Gowers1}. Specifically, we replace an indicator function by nonnegative
set weights and require a union to retain at least the smaller input weight.
The question is whether this form of closure still forces an element
to carry a constant fraction of the total weight.

To state the problem precisely, call a nonzero function
$w:2^V\to\rr_{\ge0}$ \emph{admissible} if
\begin{equation}\label{eq:admissible}
 w(\varnothing)=0,\qquad
 w(A\cup B)\ge\min\{w(A),w(B)\}
\end{equation}
for all $A, B \subseteq V$.
Write
\[
 p_x(w)=\frac{\sum_{A\ni x}w(A)}{\sum_{A \subseteq V}w(A)},\qquad
 \delta(w)=\max_{x\in V}p_x(w).
\]
For indicator weights, admissibility says exactly that the positive support
is a nonempty union-closed family not containing $\varnothing$. For general
weights, the minimum condition says that every positive superlevel family is
union-closed. The condition $w(\varnothing)=0$ rules out concentrating weight
on a set containing no coordinates.

We can ask the following question, a weighted analogue of the union-set conjecture. 

\begin{question}[Weighted union-set conjecture]
    Let $w: 2^V \to \rr_{\geq 0}$ be an admissible function. Then is it always true that $\delta (w) \geq c_w$ for some absolute constant $c_w > 0$?
\end{question}

 It is not difficult to see that any lower bound on $c_{w}$ implies a lower bound on $c_{\mathrm{us}}$. Indeed, we can simply take $w$ to be the indicator functions of sets in a union-closed family\footnote{The indicator reduction directly applies to families omitting \(\varnothing\). The same abundance bound follows for families containing \(\varnothing\) by taking disjoint-ground-set Cartesian powers, deleting the unique empty member, and letting the number of factors tend to infinity.}. Hence, proving the statement with $c_w \geq 1/2$ immediately implies \Cref{conj:frankl}.  In this paper, we prove that this stronger version of union-set conjecture is false. Namely, for any $\epsilon > 0$, there exists an admissible function over a finite ground set such that $\delta (w) < \epsilon$.

 \begin{theorem}[Main theorem]\label{thm:main}
     For every $\varepsilon>0$, there exists an admissible function over a finite ground set such that $\delta (w) < \epsilon$.
 \end{theorem}
We deduce \Cref{thm:main} from the following quantitative construction.

\begin{theorem}[Quantitative counterexample]\label{thm:example}
For every pair of integers $r\ge2$ and $m\ge1$, there exists a finite ground set
$V$ and an admissible function $w:2^V\to\rr_{\ge0}$ such that
\begin{equation}\label{eq:main} \delta(w)\le\frac1m+\frac{m-1}{mr}.
\end{equation}
\end{theorem}

\begin{proof}(of \Cref{thm:main} assuming \Cref{thm:example})
    Take $m=r\ge\max\{2,\lceil2/\varepsilon\rceil\}$ in
\Cref{thm:example}. Then
\[
 \delta(w)\le\frac{2r-1}{r^2}<\frac2r\le\varepsilon.
\]
Multiplying rational weights by a common positive integer denominator
preserves both conditions in \eqref{eq:admissible} and all the normalized
frequencies $p_x(w)$.
\end{proof}

\paragraph{Context and related constructions.}
The minimum-weight formulation was discussed in Polymath11~\cite{Gowers1};
the normalization $w(\varnothing)=0$ was explicitly proposed in the subsequent
FUNC3 discussion~\cite{Gowers3}. Element duplication, which we call
\emph{cloning}, is also described in~\cite{Gowers1}. Kahn's course problem and
solution~\cite{KahnProblem,KahnSolution} give a block construction disproving
the weighted $1/2$ statement. The example also has zero weight at the empty
set. Its counting principle is that most members of the union-closed family
belong to the original generating classes, whereas unions involving several
classes contribute relatively few additional sets.

Ellis, Ivan, and Leader~\cite{EIL} construct union-closed families in which
every element of a unique smallest member has arbitrarily small frequency.
The extension needed here has a different input requirement: it must preserve
a cloned copy of \emph{any prescribed old family}, not just exhibit a family
with one specially chosen rare set. This permits repeated extensions while
retaining all previous stages. Our proof uses Kahn's block-counting idea in
this extension-and-iteration scheme. 

\paragraph{Recent development on the union-set conjecture.}
Gilmer's 2022 breakthrough established the first positive absolute lower bound toward Frankl's union-closed sets conjecture: every finite union-closed family containing a nonempty set has an element in at least a $0.01$-fraction of its members, compared with the conjectured $1/2$. His proof uses entropy growth under the union of independent samples, introducing an information-theoretic approach to the problem~\cite{Gilmer2022}. Subsequent work of Alweiss, Huang, and Sellke and of Chase and Lovett improved the constant to $\alpha=(3-\sqrt{5})/2\approx0.381966$~\cite{AHS2022,CL2022}. Dependent-coupling arguments initiated by Sawin and further developed by Liu yielded improvements beyond this threshold~\cite{Sawin2023,Liu2023}. A September 2026 preprint of Jiang reports a computer-assisted improvement to $0.38288525$~\cite{Jiang2026}.

Recent work has also clarified the extent and limitations of these abundance phenomena. Das and Wu proved that, for every fixed $k$, the $k$th-most frequent element has frequency at least $\alpha-o(1)$ as the size of the union-closed family tends to infinity~\cite{DasWu2024}. Conversely, abundance need not occur within a distinguished small member: Ellis, Ivan, and Leader constructed families with a unique smallest set $S$ in which every element of $S$ has frequency at most $(1+o(1))\log_2|S|/(2|S|)$~\cite{EIL}. Ellis also showed that average overlap density can tend to zero, disproving another proposed constant-bound strengthening~\cite{Ellis2020}. These contrasting results motivate asking which abundance conclusions survive beyond the uniform distribution on a single union-closed family. Our result identifies a further limitation: even when $w(\varnothing)=0$ and every positive superlevel family of $w$ is union-closed, the largest normalized weighted coordinate marginal can be arbitrarily small.

\subsection{Proof overview}

Our starting point is a sequence of non-empty union-closed families $\calF_1 \subseteq \cdots \subseteq \calF_m \subseteq 2^V \backslash \{\emptyset\}$. Define the function $w: 2^V \to \rr_{\geq 0}$ as 
\[
w(A) = \frac{1}{m} \sum_{j=1}^m \frac{\mathbf{1}_{\calF_j}(A)}{|\calF_j|}.
\]
It is not hard to see that (1) $w$ is admissible (2) $\sum_{A \subseteq V} w(A) = 1$ and (3) $p_x(w) = (1/m) \sum_{j=1}^m q_{\calF_j}(x)$ where $q_{\calF_j}(x):= |\{A \in \calF_j: x \in A\}| / |\calF_j|$, i.e., the fraction of sets in $\calF_j$ that contain the element $x$. Moreover, if for each $x \in V$, there exists a label $b(x) \in [m]$ such that 
\begin{equation}\label{eqn:set_family_conditions}
     q_{\F_j}(x)=0\quad(j<b(x)),\qquad
 q_{\F_j}(x)\le1/r\quad(j>b(x)),
\end{equation}
then we would have
\[
p_x(w) \leq \frac{1}{m}\Big(1 + \frac{m-b(x)}{r} \Big) \leq \frac{1}{m} + \frac{m-1}{mr},
\]
proving \Cref{thm:example}, and hence \Cref{thm:main}. 

Hence, it suffices to construct a sequence of families $\calF_1 \subseteq \cdots \subseteq \calF_m$ satisfying (\ref{eqn:set_family_conditions}). Next we describe how this can be achieved.

\paragraph{Cloning step.} The main ingredient is an extension operation. Given any nonempty union-closed family \(\mathcal F\) of nonempty sets and an integer \(r\ge2\), we construct a union-closed family \(\mathcal G\) containing a cloned copy of \(\mathcal F\), such that every inherited coordinate has frequency at most \(1/r\) in \(\mathcal G\). Newly introduced coordinates are allowed to be frequent.

To illustrate this cloning operation, let's consider the following simple example: start with the union-closed family $\calF=\bigl\{\{a\},\{b\},\{a,b\}\bigr\}.$ Replace \(a\) by two copies \(a_1,a_2\), and \(b\) by two copies \(b_1,b_2\). On the new ground set $Y = \{a_1,a_2,b_1,b_2\}$, the cloned family is
\[
E(\mathcal F) = \bigl\{ \{a_1,a_2\}, \{b_1,b_2\}, \{a_1,a_2,b_1,b_2\} \bigr\}.
\]

There are still three sets, and every copy still has frequency \(2/3\). Thus cloning alone has not decreased any coordinate frequency. The benefit appears when new sets may choose the copies independently. Let \(C\subseteq Y\) be a freely chosen portion of a new set, with any other coordinates held fixed. There are \(2^4=16\) possible choices of \(C\). However, after taking the union with the cloned old set $E(\{a\})=\{a_1,a_2\}$,
the possible outputs are precisely

$$ \bigl\{ \{a_1,a_2\}\cup D: D\subseteq\{b_1,b_2\} \bigr\}, $$

 We only have \(2^2=4\) sets. Indeed, once we take the union with \(E(\{a\})\), the presence or absence of \(a_1\) and \(a_2\) in \(C\) no longer affects the output. Each output therefore comes from four different inputs. Sixteen choices of the new set’s freely varying portion produce only four distinct unions with the old set.

With \(t\) copies of each original coordinate, the same calculation sends \(2^{2t}\) choices of \(C\) to just \(2^t\) distinct unions:

$$ \frac{\text{number of distinct union outputs}} {\text{number of freely chosen inputs}} = \frac{2^t}{2^{2t}} = 2^{-t}. $$

Thus cloning preserves the old instance while making many new inputs collapse to the same union output. In the full construction, this allows us to add large principal classes while keeping the additional sets forced by unions with the old family comparatively few.
In addition, we will show that this operation also preserves unions, the number of old sets, and the frequency inherited by each copy. 

More concretely, write \(n=|V|\) for the old ground-set size. Form \(r\) disjoint blocks \(Y_1,\ldots,Y_r\), each containing \(t\) copies of every old coordinate. The cloned set \(E(A)\) contains all copies of the coordinates in \(A\), and hence occupies at least \(t\) positions in every \(Y_j\) whenever \(A\ne\varnothing\). Introduce auxiliary blocks \(X_1,\ldots,X_r\), each of size \(s=nt+t\), and write \(X=\bigsqcup_jX_j\). For each \(j\), add the principal class

$$ \mathcal P_j= \left\{ X_j\cup B\cup C: B\subseteq X\setminus X_j,\ C\subseteq Y_j \right\}. $$

Each class has \(L=2^{(r-1)s+nt}\) members. An inherited coordinate in \(Y_j\) occurs in half of \(\mathcal P_j\) and in none of the other principal classes. The principal classes are nearly disjoint, suggesting an inherited-coordinate frequency close to \(1/(2r)\).

\paragraph{Principal classes counting and iterating.} Let \(\mathcal G\) be the union-closure of the principal classes together with \(E(\mathcal F)\). Two complementary mechanisms keep the additional sets few. First, unioning a member of \(\mathcal P_j\) with a fixed nonempty old set forces at least \(t\) previously free positions in \(Y_j\), leaving at most \(L2^{-t}\) distinct outputs. Second, combining \(k\) different principal classes forces \(k\) auxiliary blocks to be full. Each additional class gains \(nt\) free old-coordinate positions but loses \(s=nt+t\) free auxiliary positions, leaving only \(L2^{-(k-1)t}\) possible outputs for a fixed collection of classes. Adding a fixed old set cannot increase this count.

Every generated set is a union of at most one old set and at most one member of each principal class, since the old family and each principal class are themselves union-closed. These bounds therefore cover the entire union-closure. Choosing \(t\) sufficiently large makes the number of non-principal sets at most \(L/4\), while the principal classes supply at least \(3rL/4\) distinct sets. Even allowing every non-principal set to contain a given inherited coordinate \(y\), we obtain

$$ \frac{|\{S\in\mathcal G:y\in S\}|}{|\mathcal G|} \le \frac{L/2+L/4}{3rL/4} = \frac1r. $$

Finally, starting from a one-set family on one coordinate, we repeatedly apply the extension operation. At every step, we apply the same cloning map to all earlier families, not just to the largest one. Their nesting and coordinate frequencies are therefore preserved. Copies retain their ancestor’s generation, while the new auxiliary coordinates form the next generation.

\paragraph{Organization.} \Cref{sec:reduction_to_union_closed_families} reduces the desired construction of weights to the construction of a sequence of union-closed set families. \Cref{sec:cloning} establishes the cloning properties.
\Cref{sec:principal} counts unions of principal classes, and
\Cref{sec:counting} proves the rare-coordinate extension through
separate counting lemmas. Lastly, \Cref{sec:iterating} then checks that the
entire earlier chain survives each extension. 

\paragraph{Notation.}
All ground sets are finite. A family is \emph{empty-free} if it does not
contain $\varnothing$; this does not assert that the family itself is nonempty.
For a nonempty family $\Hh\subseteq2^V$ and $x\in V$, let
\[
 \Hh_x=\{A\in\Hh:x\in A\},\qquad
 q_{\Hh}(x)=\frac{|\Hh_x|}{|\Hh|}.
\]
We use $[k]=\{1,\ldots,k\}$, and $\mathbf{1}_{\Hh}$ denotes the indicator of a
family. A family generated by other families consists of their
\emph{nonempty finite unions}; we never adjoin the empty union by convention.

\section{Reduction to union-closed set families}\label{sec:reduction_to_union_closed_families}
In this section, we define the weight function $w$ over nested union-closed families. Throughout this section, fix nonempty union-closed families
\[
 \F_1\subseteq\cdots\subseteq\F_m\subseteq2^V\setminus\{\varnothing\},
\]
and define
\begin{equation}\label{eq:mixture}
 w(A)=\frac1m\sum_{j=1}^m\frac{\mathbf{1}_{\F_j}(A)}{|\F_j|}.
\end{equation}

\begin{lemma}[Weights on nested families]\label{lem:reduce_to_nested}
    Let $m \geq 1$ and let $\calF_1 \subseteq ... \subseteq \calF_m \subseteq 2^V \backslash \{\emptyset\}$ be nonempty union-closed families supported on a common finite ground set $V$.  Let $w:2^V \to \rr_{\geq 0}$ be defined as in (\ref{eq:mixture}). Then we have the following:
    \begin{enumerate}[(a)]
        \item $w$ is admissible
        \item $\sum_{A \subseteq V}w(A) = 1$ and
        \item $p_x (w) = \frac{1}{m} \sum_{j=1}^mq_{\calF_j}(x)$ for all $x \in V$.
    \end{enumerate}
\end{lemma}
\begin{proof}
    \begin{enumerate}[(a)]
        \item Since $\calF_m$ does not contain $\emptyset$ by assumption, we know $w(\emptyset) = 0$. Take arbitrary $A, B \subseteq V$ with arbitrary weights. If either weight is zero, the minimum condition follows from nonnegativity.Otherwise, let $a \in [m]$(resp. $b \in [m]$) be the smallest number such that $A \in \calF_a$ (resp. $B \in \calF_b$). Exchanging if necessary, we may assume $a \leq b$. Then by nesting, we know $w(A) \geq w(B)$. Now for every $j \geq b$, both $A, B$ belong to $\calF_j$. Since $\calF_j$ is union-closed, we know $A \cup B \in \calF_j$. Therefore, 
        \[
        w(A \cup B) \geq \frac{1}{m}\sum_{j=b}^m \frac{1}{|\calF_j|} = w(B) = \min \{w(A), w(B)\}.
        \]
        \item Follows immediately since $w(A)$ is the probability of obtaining $A$ via the following procedure: Choose $J$ uniformly from $[m]$ and then choose a set uniformly from $\calF_J$.
        \item By definition, for any $x \in V$
        \[
        p_x(w) =  \sum_{A \ni x}w(A) = \sum_{A \ni x} \frac{1}{m}\Big(\sum_{j=1}^m \frac{\mathbf{1}_{\calF_j}(A)}{|\calF_j|}\Big) = \frac{1}{m} \sum_{j=1}^m q_{\calF_j}(x),
        \]
        where the first equality holds since $\sum_{A \subseteq V}w(A) = 1$ and the last equality follows by switching the order of summation.
    \end{enumerate}
\end{proof}

\begin{lemma}[One unrestricted stage per coordinate]\label{lem:generation-average}
Suppose $r\ge2$ and each $x\in V$ has a label $b(x)\in[m]$ satisfying
\begin{equation}\label{eq:generation-pattern}
 q_{\F_j}(x)=0\quad(j<b(x)),\qquad
 q_{\F_j}(x)\le1/r\quad(j>b(x)).
\end{equation}
Then the weights in \eqref{eq:mixture} satisfy
$\delta(w)\le1/m+(m-1)/(mr)$.
\end{lemma}
\begin{proof}
Think of $b(x)$ as the generation in which $x$ was introduced. The only
unrestricted frequency is at that stage, where the trivial bound is $1$.
 If $b(x)=k$,
Lemma~\ref{lem:reduce_to_nested} gives
\[
 p_x(w)\le\frac1m\left(1+\frac{m-k}{r}\right)
          \le\frac1m+\frac{m-1}{mr}.
\]
The parameter $m$ dilutes the one potentially large contribution, while $r$
controls the accumulated contributions from later stages.
\end{proof}

\section{Cloning and unions with old sets}\label{sec:cloning}
In this section, we define the cloning map and prove a few properties of the map. A \emph{cloning map} replaces each coordinate $v$ by a disjoint nonempty
block $E_v$ and maps $A$ to $\bigcup_{v\in A}E_v$. Our maps have the following
explicit form. Fix a ground set $V$ of size $n\ge1$ and integers $r\ge2$,
$t\ge1$. Form
\begin{equation}\label{eq:clone}
 Y_j=V\times\{j\}\times[t],\qquad
 Y=\bigsqcup_{j=1}^rY_j,\qquad E(A)=A\times[r]\times[t].
\end{equation}
Thus every old coordinate has $t$ copies in each of $r$ blocks. We verify
separately that cloning preserves the old structure and that it reduces
the number of distinct unions with an old set. Define \(E(\mathcal H):=\{E(A):A\in\mathcal H\}\).

\begin{lemma}[Cloning preserves the set operations]\label{lem:clone-structure}
The map $E$ is injective, sends $\varnothing$ to $\varnothing$, and satisfies
$E(A\cup B)=E(A)\cup E(B)$. In particular, applying $E$ to every family in a
nested union-closed chain preserves nesting and union closure.
\end{lemma}
\begin{proof}
For a copy $y=(v,j,\ell)$, the statement $y\in E(A)$ is equivalent to $v\in A$.
Because every coordinate has a copy, different old sets have different
images. The same membership rule proves the union identity and the
empty-set assertion. The identity proves preservation of union closure,
and applying any map elementwise to families preserves their inclusions.
\end{proof}

\begin{lemma}[Every copy has its original frequency]\label{lem:clone-frequency}
For any nonempty family $\Hh\subseteq2^V$ and any $y=(v,j,\ell)\in Y$,
\[
 |E(\Hh)|=|\Hh|,\qquad q_{E(\Hh)}(y)=q_{\Hh}(v).
\]
\end{lemma}
\begin{proof}
Injectivity gives the first equality. The equivalence
$y\in E(A)\Longleftrightarrow v\in A$ gives
$|E(\Hh)_y|=|\Hh_v|$, and division proves the second.
Cloning therefore does not itself make a coordinate rare. It preserves all
previous frequencies while changing how old sets interact with sets to be
introduced next.
\end{proof}

\begin{lemma}[An old set collapses free choices]\label{lem:clone-collapse}
For every nonempty $A\subseteq V$ and every $j\in[r]$,
\begin{equation}\label{eq:clone-collapse}
 \left|\{E(A)\cup C:C\subseteq Y_j\}\right|
 =2^{(n-|A|)t}\le2^{nt-t}.
\end{equation}
\end{lemma}
\begin{proof}
Before taking the union, there are $2^{nt}$ choices for $C$. Outside $Y_j$
the output is fixed. Inside $Y_j$, all $|A|t$ positions of $E(A)\cap Y_j$
are forced to be present, regardless of their status in $C$. The remaining
$(n-|A|)t$ positions can vary arbitrarily, and every such pattern is attained.
This proves the equality. Since $A\ne\varnothing$, at least $t$ positions
are forced, giving the inequality.
\end{proof}

\section{Principal classes and auxiliary blocks}\label{sec:principal}
Use the same notations as in \Cref{sec:cloning}. Introduce disjoint sets
$X_1,\ldots,X_r$, disjoint from $Y$, with
\[
 |X_j|=s:=(n+1)t,\qquad X=\bigsqcup_{j=1}^rX_j.
\]
Define the principal classes and their common size parameter by
\begin{equation}\label{eq:principal}
 \Pp_j=\{X_j\cup B\cup C:B\subseteq X\setminus X_j,
                               \ C\subseteq Y_j\},\qquad
 L=2^{(r-1)s+nt}.
\end{equation}
An additional class will free $nt$ old-coordinate positions but force $s$
auxiliary positions. The difference $s-nt=t$ will control mixed unions.

\begin{lemma}[Size and closure of a principal class]\label{lem:principal-size}
Each $\Pp_j$ is an empty-free union-closed family of size $L$.
\end{lemma}
\begin{proof}
It is easy to see $\calP_j$ is union-closed. On the other hand,
\[
|\calP_j| = 2^{|X \backslash X_j|} \times 2^{|Y_j|} = 2^{(r-1)s+nt}
\]
\end{proof}

\begin{lemma}[Old-coordinate incidences in the principal part]\label{lem:principal-degree}
If $\Pp=\bigcup_{i=1}^r\Pp_i$, then for every $y\in Y$,
\begin{equation}\label{eq:principal-degree}
 |\Pp_y|=L/2.
\end{equation}
\end{lemma}
\begin{proof}
A coordinate $y\in Y_j$ is freely chosen in $\Pp_j$, so exactly half that
class contains it. It is absent from every other principal class. Hence
there is no overlap to subtract among sets containing $y$, and the count
is exactly $L/2$.
\end{proof}

\begin{lemma}[Unions from a fixed collection of classes]\label{lem:class-unions}
For nonempty $I\subseteq[r]$ with $|I|=k$, let
$\U_I=\{\bigcup_{j\in I}P_j:P_j\in\Pp_j\}$. Then
\begin{equation}\label{eq:class-unions}
 |\U_I|=2^{(r-k)s+knt}=L2^{-(k-1)t}.
\end{equation}
\end{lemma}
\begin{proof}
In any output, all $X_j$ with $j\in I$ are full. The remaining $X$-positions
are free, and the old part is an arbitrary subset of
$\bigcup_{j\in I}Y_j$. Conversely, every set with this description is
attainable: choose its $Y_j$-part in the corresponding $P_j$, and put its
desired $X$-positions outside the forced blocks into one chosen principal
set. Thus there are $(r-k)s+knt$ free positions. Since $s-nt=t$,
\[
 (r-k)s+knt=((r-1)s+nt)-(k-1)t,
\]
which proves the formula.
\end{proof}

\section{Counting the extension}\label{sec:counting}
Now let $\varnothing\ne\F\subseteq2^V\setminus\{\varnothing\}$ be the old
union-closed family, and set $M=|\F|$. Use the blocks above and let $\G$ be
all nonempty finite unions of members of
\begin{equation}\label{eq:generators}
 E(\F)\cup\Pp_1\cup\cdots\cup\Pp_r.
\end{equation}
Write $\Pp=\bigcup_j\Pp_j$ and
\begin{equation}\label{eq:remainder}
 R=|\G\setminus\Pp|.
\end{equation}
We keep $t$ arbitrary until \Cref{lem:remainder-small}. First we identify
all possible outputs, then lower-bound the principal contribution, and then
bound the different sources of additional sets.

\begin{lemma}[A normal form for every generated set]\label{lem:normal-form}
Every $S\in\G$ has a representation
\begin{equation}\label{eq:normal-form}
 S=E(A)\cup\bigcup_{j\in I}P_j,
 \qquad A\in\F\cup\{\varnothing\},\quad I\subseteq[r],\quad P_j\in\Pp_j,
\end{equation}
where $(A,I)\ne(\varnothing,\varnothing)$. Conversely, every such union lies
in $\G$.
\end{lemma}
\begin{proof}
Combine all old generators in a representation of $S$ into one. Their union
has the form $E(A)$ with $A\in\F$, since $E$ preserves unions and $\F$ is
union-closed. If no old generator is used, set $A=\varnothing$. Also combine
all generators from a given class into one member of that class, using
\Cref{lem:principal-size}. This gives \eqref{eq:normal-form}.
The converse is exactly the definition of $\G$.

\end{proof}

Notice that the representation need not be unique. Uniqueness is unnecessary: all the
counts of additional sets below are upper bounds, so multiple counting only
makes them more conservative.

\begin{lemma}[Many distinct principal sets]\label{lem:denominator}
The completed family satisfies
\begin{equation}\label{eq:denominator}
 |\G|\ge rL(1-2^{-nt}).
\end{equation}
\end{lemma}
\begin{proof}
Within $\Pp_j$, retain only the sets whose $Y_j$-part is nonempty. There are
$2^{(r-1)s}(2^{nt}-1)=L(1-2^{-nt})$ such sets. The retained classes are
pairwise disjoint, because their nonempty old parts are supported in
different $Y$-blocks. Summing over $j$ gives the result.
\end{proof}

\begin{lemma}[One principal class joined to old sets]\label{lem:one-class-old}
The family
\[
 \Hh_1=\{E(A)\cup P:A\in\F,\ j\in[r],\ P\in\Pp_j\}
\]
has size at most $rML2^{-t}$.
\end{lemma}
\begin{proof}
Fix $A\in\F$ and $j\in[r]$. In $E(A)\cup P$, the $X$-part has
$2^{(r-1)s}$ choices. \Cref{lem:clone-collapse} counts the possible
old parts, so the exact number of outputs for this fixed pair is
\[
 2^{(r-1)s+(n-|A|)t}=L2^{-|A|t}\le L2^{-t}.
\]
There are $rM$ choices of $(A,j)$. A union bound proves the statement.
\end{proof}

\begin{lemma}[Outputs using at least two principal classes]\label{lem:several-classes-old}
Let $\Hh_{\ge2}$ consist of all sets of the form \eqref{eq:normal-form}
with $|I|\ge2$. Then
\begin{equation}\label{eq:several-bound}
 |\Hh_{\ge2}|\le(M+1)2^rL2^{-t}.
\end{equation}
\end{lemma}
\begin{proof}
For fixed $A$ and $I$ with $|I|=k$, the map $U\mapsto E(A)\cup U$ cannot
increase the number of outputs. \Cref{lem:class-unions} therefore bounds
that number by $L2^{-(k-1)t}$. There are $M+1$ choices of $A$, including
$\varnothing$, and $\binom rk$ choices of $I$. Hence
\[
 \begin{split}
 |\Hh_{\ge2}|
 &\le(M+1)L\sum_{k=2}^r\binom rk2^{-(k-1)t}\\
 &\le(M+1)L2^{-t}\sum_{k=2}^r\binom rk
 \le(M+1)2^rL2^{-t}.
 \end{split}
\]
\end{proof}

\begin{lemma}[The additional sets are negligible]\label{lem:remainder-small}
With
\begin{equation}\label{eq:parameters}
 t=n+r+3,\qquad s=(n+1)t,
\end{equation}
the remainder in \eqref{eq:remainder} satisfies $R\le L/4$.
\end{lemma}
\begin{proof}
By \Cref{lem:normal-form}, a non-principal set either uses no principal
class, exactly one class and an old set, or at least two classes. The first
case has $M$ possible outputs. \Cref{lem:one-class-old}
and \Cref{lem:several-classes-old} give
\begin{equation}\label{eq:remainder-raw}
 R\le M+\bigl[rM+(M+1)2^r\bigr]L2^{-t}.
\end{equation}
This is an exhaustive cover, not a claimed disjoint partition.

Since $\F$ is a nonempty family of nonempty subsets, $n\ge1$ and $M+1\le2^n$.
Also $L\ge2^t$ and $r+1\le2^r$ for $r\ge2$. Dividing
\eqref{eq:remainder-raw} by $L$ and absorbing $M/L$ gives
\[
 \begin{split}
 \frac RL
 &\le\bigl[(r+1)M+(M+1)2^r\bigr]2^{-t}\\
 &\le2^{n+r+1-t}=\frac14.
 \end{split}
\]
Thus the factor $2^{-t}$ absorbs all choices of old sets and principal
classes. The constants are deliberately loose: smallness relative to one
principal class is sufficient.
\end{proof}

\begin{lemma}[Rare-coordinate extension]\label{lem:extension}
For every nonempty empty-free union-closed family $\F\subseteq2^V$ and
integer $r\ge2$, there are disjoint finite sets $Y,X$, a cloning map
$E:2^V\to2^Y$, and an empty-free union-closed family $\G\subseteq2^{Y\sqcup X}$
such that
\begin{equation}\label{eq:extension}
 E(\F)\subseteq\G,\qquad q_\G(y)\le1/r\quad(y\in Y).
\end{equation}
\end{lemma}
\begin{proof}
Use \eqref{eq:parameters} in the construction of \eqref{eq:generators}.
The family is union-closed by construction and contains $E(\F)$. All its
generators are nonempty, so all its members are nonempty.

For $y\in Y$, \Cref{lem:principal-degree} counts $L/2$ principal sets
containing $y$. Even if every additional set contains $y$,
\Cref{lem:remainder-small} yields
\[
 |\G_y|\le L/2+R\le3L/4.
\]
On the other hand, $nt\ge2$, so \Cref{lem:denominator} gives
$|\G|\ge3rL/4$. Consequently,
\[
 q_\G(y)\le\frac{3L/4}{3rL/4}=\frac1r.
\]
The estimate is simultaneous for all inherited coordinates. No bound is
required on the new coordinates in $X$.
\end{proof}

\paragraph{Remark.} The operation embeds the old family after cloning; it does not preserve the old ground set literally. This is sufficient because
\Cref{lem:clone-frequency} preserves each old frequency exactly in the
embedded family. The new ground-set size is
\begin{equation}\label{eq:one-step-size}
 |Y\sqcup X|=rnt+rs=r(2n+1)(n+r+3).
\end{equation}
Thus the operation is finite for every finite input. At the next step,
today's new coordinates become inherited coordinates and can be made rare.

\section{Iteration while preserving the whole chain}\label{sec:iterating}
Fix an integer $r \geq 2$. For a chain $\F_1\subseteq\cdots\subseteq\F_k$ on a common
finite ground set, we say a generation label $b:V\to[k]$ is \emph{valid} if
\begin{equation}\label{eq:valid-labels}
 q_{\F_j}(x)=0\quad(j<b(x)),\qquad
 q_{\F_j}(x)\le1/r\quad(j>b(x)).
\end{equation}
This is exactly the pattern used in \Cref{lem:generation-average}.
We separate one chain extension from the check of its generation labels.

\begin{lemma}[Extend the chain without changing its earlier frequencies]
\label{lem:extend-chain}
Let $\F_1\subseteq\cdots\subseteq\F_k$ be nonempty empty-free union-closed
families on $V$. Apply \Cref{lem:extension} to $\F_k$, obtaining $E,Y,X,\G$.
Set
\[
 \F'_j=E(\F_j)\quad(1\le j\le k),\qquad \F'_{k+1}=\G.
\]
These form a nested nonempty empty-free union-closed chain. For $j\le k$,
\begin{equation}\label{eq:earlier-frequencies}
 q_{\F'_j}(y)=q_{\F_j}(v)\quad\text{if $y$ is a copy of $v$},\qquad
 q_{\F'_j}(z)=0\quad(z\in X).
\end{equation}
\end{lemma}
\begin{proof}
\Cref{lem:clone-structure} preserves all earlier inclusions and union
closure. The final inclusion is $E(\F_k)\subseteq\G$, supplied by the
extension lemma. Nonemptiness and empty-freeness are preserved as well.
\Cref{lem:clone-frequency} proves the first equality in
\eqref{eq:earlier-frequencies}. Every earlier image is contained in $Y$, so
new coordinates in $X$ are absent, proving the second.

It is important to apply the same map $E$ to \emph{every} earlier family,
not just to $\F_k$. This puts the whole history on one new ground set without
changing any of its frequency data.
\end{proof}

\begin{lemma}[Generation labels survive an extension]\label{lem:labels}
If $b$ is valid for the chain in \Cref{lem:extend-chain}, then a valid
labeling for the enlarged chain is
\[
 b'(y)=b(v)\quad\text{for every copy $y$ of $v$},\qquad
 b'(z)=k+1\quad(z\in X).
\]
\end{lemma}
\begin{proof}
For an inherited coordinate, all frequencies at stages $j\le k$ are
unchanged by \eqref{eq:earlier-frequencies}. Hence all the old requirements
remain valid. At the added stage $k+1$, the extension lemma gives frequency
at most $1/r$. Since its label is at most $k$, this is exactly a required
later-stage bound.

A new coordinate has label $k+1$ and is absent at every earlier stage by
\eqref{eq:earlier-frequencies}. There is no constraint beyond the trivial
upper bound $1$ at its own stage, and there are no later stages yet. Thus
\eqref{eq:valid-labels} also holds for all new coordinates.
\end{proof}

\begin{lemma}[Arbitrarily long chains with valid generations]\label{lem:iterated-chain}
For every integer $m\ge1$, there are nonempty empty-free union-closed
families $\F_1\subseteq\cdots\subseteq\F_m$ on a finite common ground set
and a valid generation labeling $b:V\to[m]$.
\end{lemma}
\begin{proof}
For $m=1$, take $V_1=\{v\}$, $\F_1=\{\{v\}\}$, and $b(v)=1$. The validity
conditions are trivially true. Given a chain of length $k$, extend it using
\Cref{lem:extend-chain}, and extend its labels using
\Cref{lem:labels}. This proves the assertion by induction.

For completeness, if $n_k$ is the ground-set size after stage $k$, the
particular construction above satisfies
\[
 n_1=1,\qquad n_{k+1}=r(2n_k+1)(n_k+r+3).
\]
Thus every stage is finite. The ambient ground set changes at each step,
but all earlier families are transported together. The final chain lives
on the single finite ground set obtained after the last extension.
\end{proof}

\begin{proof}(of \Cref{thm:example})
    Fix integers \(r\ge2\) and \(m\ge1\). By Lemma 6.3, there exist nonempty empty-free union-closed families

$$ \mathcal F_1\subseteq\cdots\subseteq\mathcal F_m $$

on a common finite ground set, with a valid generation labeling. Define \(w\) by (2.1). Lemma 2.1 gives admissibility and normalization, while Lemma 2.2 yields

$$ \delta(w)\le\frac1m+\frac{m-1}{mr}. $$
\end{proof}

\paragraph{Acknowledgment. }The author would like to thank Raghu Meka and Shachar Lovett for helpful discussions. The author initiated discussion with GPT 5.5 / 5.6 with the aim of proving the conjecture. Through the discussions that last for months, a lot of back-and-forth conversations were carried out. The goal was then shifted to disproving the conjecture. GPT 6 built upon idea of the author and completed the construction. Most of the writings here are done by the author. GPT 6 helped with polishing the writeup.

\bibliographystyle{alpha}
\bibliography{refs}

@article{EIL,
  author        = {Ellis, David and Ivan, Maria-Romina and Leader, Imre},
  title         = {Small sets in union-closed families},
  journal       = {The Electronic Journal of Combinatorics},
  volume        = {30},
  number        = {1},
  pages         = {P1.8},
  year          = {2023},
  doi           = {10.37236/11004},
  eprint        = {2201.11484},
  archivePrefix = {arXiv},
  url           = {https://doi.org/10.37236/11004}
}

@misc{Gowers1,
  author       = {Gowers, Timothy},
  title        = {{FUNC1}---strengthenings, variants, potential counterexamples},
  howpublished = {Gowers's Weblog},
  year         = {2016},
  month        = jan,
  url          = {https://gowers.wordpress.com/2016/01/29/func1-strengthenings-variants-potential-counterexamples/},
  note         = {Posted January 29, 2016. See ``Another weighted version''
                  and ``Duplicating elements''}
}

@misc{Gowers3,
  author       = {Gowers, Timothy and {contributors}},
  title        = {{FUNC3}---further strengthenings and variants},
  howpublished = {Gowers's Weblog},
  year         = {2016},
  month        = feb,
  url          = {https://gowers.wordpress.com/2016/02/13/func3-further-strengthenings-and-variants/},
  note         = {Posted February 13, 2016. See the comments of February 15--17
                  on the minimum-weight condition and $w(\varnothing)=0$}
}

@misc{KahnProblem,
  author       = {Kahn, Jeff},
  title        = {{642.583 Problem Set 5}},
  howpublished = {Course material, Rutgers University},
  url          = {https://sites.math.rutgers.edu/~jkahn/p5.pdf},
  note         = {Available at https://sites.math.rutgers.edu/~jkahn/p5.pdf}
}

@misc{KahnSolution,
  author       = {Kahn, Jeff},
  title        = {{583 PS5 Solutions}},
  howpublished = {Course material, Rutgers University},
  url          = {https://sites.math.rutgers.edu/~jkahn/s5.pdf},
  note         = {Available at https://sites.math.rutgers.edu/~jkahn/s5.pdf}
}

@misc{Gilmer2022,
  author        = {Gilmer, Justin},
  title         = {A constant lower bound for the union-closed sets conjecture},
  year          = {2022},
  eprint        = {2211.09055},
  archivePrefix = {arXiv},
  primaryClass  = {math.CO},
  doi           = {10.48550/arXiv.2211.09055},
  url           = {https://arxiv.org/abs/2211.09055}
}

@misc{Sawin2023,
  author        = {Sawin, Will},
  title         = {An improved lower bound for the union-closed set conjecture},
  year          = {2023},
  eprint        = {2211.11504},
  archivePrefix = {arXiv},
  primaryClass  = {math.CO},
  doi           = {10.48550/arXiv.2211.11504},
  url           = {https://arxiv.org/abs/2211.11504},
  note          = {Version 3, revised June 19, 2023}
}

@misc{AHS2022,
  author        = {Ryan Alweiss and Brice Huang and Mark Sellke},
  title         = {Improved Lower Bound for {Frankl}'s Union-Closed Sets Conjecture},
  year          = {2022},
  eprint        = {2211.11731},
  archivePrefix = {arXiv},
  primaryClass  = {math.CO},
  howpublished  = {arXiv:2211.11731},
  url           = {https://arxiv.org/abs/2211.11731}
}

@misc{CL2022,
  author        = {Zachary Chase and Shachar Lovett},
  title         = {Approximate union closed conjecture},
  year          = {2022},
  eprint        = {2211.11689},
  archivePrefix = {arXiv},
  primaryClass  = {math.CO},
  howpublished  = {arXiv:2211.11689},
  url           = {https://arxiv.org/abs/2211.11689}
}

@misc{Liu2023,
  author        = {Jingbo Liu},
  title         = {Improving the Lower Bound for the Union-closed Sets Conjecture via Conditionally {IID} Coupling},
  year          = {2023},
  eprint        = {2306.08824},
  archivePrefix = {arXiv},
  primaryClass  = {cs.IT},
  howpublished  = {arXiv:2306.08824},
  url           = {https://arxiv.org/abs/2306.08824}
}

@misc{Jiang2026,
  author        = {Yunjiang Jiang},
  title         = {Entropy bounds and global couplings for union-closed families},
  year          = {2026},
  eprint        = {2609.08291},
  archivePrefix = {arXiv},
  primaryClass  = {math.CO},
  howpublished  = {arXiv:2609.08291},
  note          = {Preprint, submitted 8 September 2026},
  url           = {https://arxiv.org/abs/2609.08291}
}

@misc{DasWu2024,
  author        = {Shagnik Das and Saintan Wu},
  title         = {Frequent elements in union-closed set families},
  year          = {2024},
  eprint        = {2412.03862},
  archivePrefix = {arXiv},
  primaryClass  = {math.CO},
  howpublished  = {arXiv:2412.03862},
  note          = {Version 3, revised 11 July 2025},
  url           = {https://arxiv.org/abs/2412.03862}
}

@misc{Ellis2020,
  author        = {David Ellis},
  title         = {Union-closed families with small average overlap densities},
  year          = {2020},
  eprint        = {2012.03235},
  archivePrefix = {arXiv},
  primaryClass  = {math.CO},
  howpublished  = {arXiv:2012.03235},
  url           = {https://arxiv.org/abs/2012.03235}
}

@incollection{frankl1995extremal,
  author    = {Frankl, Peter},
  title     = {Extremal Set Systems},
  booktitle = {Handbook of Combinatorics},
  volume    = {2},
  pages     = {1293--1329},
  year      = {1995},
  publisher = {Elsevier}
}

\end{document}